\documentclass{amsart}


\usepackage{amsmath,amssymb, amstext,amsthm,amsfonts}
\usepackage[dvips]{graphicx}

\thanks{Partially supported by CNPq, FAPERJ and PRONEX/DS from Brazil.}

\newtheorem{theorem}{Theorem}[section] 
\newtheorem{lemma}[theorem]{Lemma}     
\newtheorem{corollary}[theorem]{Corollary}

\newtheorem{defi}[theorem]{Definition}

\title[On the essential hyperbolicity of sectional-Anosov flows]
 {On the essential hyperbolicity of sectional-Anosov flows} 

\author{S. Bautista, C. A. Morales}

\address{Departamento de Matem\'aticas, Universidad Nacional de Colombia, Bogot\'a, Colombia.}

\address{Instituto de Matem\'atica, Universidade Federal do Rio de Janeiro, P. O. Box 68530, 21945-970 Rio
de Janeiro, Brazil.}
\email{sbautistad@unal.edu.co, morales@impa.br}

\thanks{SB was partially supported by the Universidad Nacional de Colombia, Bogot\'a, Colombia.\\
CAM was partially supported by CNPq, FAPERJ and PRONEX/Dynam. Sys. from Brazil
and the Universidad Nacional de Colombia from Colombia.
CAM would like to thank the Universidad Nacional de Colombia, Bogot\'a, Colombia, for its kindly hospitality during the preparation of this paper.}

\begin{document}
\maketitle

\begin{abstract}
We prove that every sectional-Anosov flow of a compact $3$-manifold $M$ exhibits
a finite collection of hyperbolic attractors and singularities whose basins form a dense subset of $M$.
Applications to the dynamics of sectional-Anosov flows on compact $3$-manifolds include a
characterization of essential hyperbolicity,
sensitivity to the initial conditions (improving \cite{ams}) and
a relationship between the topology of the ambient manifold and the denseness
of the basin of the singularities.
\end{abstract}

\section{Introduction}

\noindent
A smooth vector field on a differentiable manifold is called {\em essentially hyperbolic} if
it exhibits a finite collection of hyperbolic attractors whose basins form an open and dense subset of
the manifold \cite{a}, \cite{cp}.
Basic examples are the Axiom A ones (by the spectral decomposition theorem \cite{kh}, \cite{pt}), 
including the {\em Anosov flows}, but not the
{\em geometric Lorenz attractor} \cite{abs}, \cite{gw}.
On the other hand, there is a class of systems, the {\em sectional-Anosov flows} \cite{m1}, whose
representative examples are
the Anosov flows, the geometric Lorenz attractors, the saddle-type hyperbolic attracting sets,
the {\em multidimensional Lorenz attractors} \cite{bpv} and the examples in
\cite{mm}, \cite{mmm}.
They motivate search of necessary and sufficient conditions for
a sectional-Anosov flows to be essentially hyperbolic, or, if there is a sort of essential hyperbolicity for them.
At first glance it is tempting to say that for every sectional-Anosov flow of a compact manifold there is a finite collection of
sectional-hyperbolic attractors whose basins form an open and dense subset.
However, this is false even in dimension three as shown \cite{bm}, \cite{bmp}.
Nevertheless, as proved in \cite{cm},
every vector field close to a transitive sectional-Anosov flow with singularities of a compact $3$-manifold
satisfies that generic points have a singularity in their omega-limit set.
This was improved later in \cite{ams} by proving that all such vector fields
satisfy that {\em the basin of the singularities is dense in the manifold}.
In general, we can
combine \cite{ams} and \cite{mp} to obtain that, for every compact $3$-manifold $M$,
there is a $C^1$ open and dense subset of sectional-Anosov vector fields all of
whose elements exhibit a finite collection of hyperbolic attractors and singularities
whose basins form a dense subset of $M$.
In this paper we strengthen this last assertion by proving that
{\em every} sectional-Anosov flow of {\em every} compact $3$-manifold $M$ exhibits a finite collection of hyperbolic attractors
and singularities whose basins form a dense subset of $M$.
This fact has some consequences in the study of the dynamics of the sectional-Anosov flows $X$
on compact $3$-manifolds.
The first one is that $X$ is essentially hyperbolic if and only if
the basin of its set of singularities is nowhere dense.
Another application is related to a result in \cite{ams} asserting that
every vector field of a compact $3$-manifold that is
$C^1$ close to a nonwandering sectional-Anosov flow is sensitive to the initial conditions.
Indeed, we extend this result by proving that {\em every} sectional-Anosov flows on {\em every} compact $3$-manifold
is sensitive to the initial conditions.
Finally, we prove that
every sectional-Anosov flow with singularities (all Lorenz-like)
but without null homotopic periodic orbits  of a compact atoroidal $3$-manifold $M$
satisfies that the basin of the set of singularities is dense in $M$.
Let us state our results in a precise way.

Consider a compact manifold  $M$ with possibly nonempty boundary $\partial M$.
To indicate its dimension $n$ we will call it $n$-manifold.
Consider also a vector field $X$ with induced flow $X_t$ on $M$,
inwardly transverse to $\partial M$ if $\partial M\neq\emptyset$ (all vector fields in this paper will be assumed to be $C^1$).
Define the {\em maximal invariant set} of $X$
$$
M(X)=\displaystyle\bigcap_{t\geq0}X_t(M).
$$
We say that  $\Lambda\subset M(X)$ is {\em invariant}
if $X_t(\Lambda)=\Lambda$ for every $t\in \mathbb{R}$.
Given $x\in M$ we define the {\em omega-limit set},
$$
\omega(x)=\left\{y\in M:y=\lim_{k\to\infty}X_{t_k}(x)\mbox{ for some sequence }t_k\to\infty\right\}.
$$
Define the {\em basin (of attraction)} of any subset $B\subset M$ as the set of points $x\in M$ such that
$\omega(x)\subset B$.
An invariant set $\Lambda$ is {\em transitive} if $\Lambda=\omega(x)$ for some $x\in\Lambda$.
An {\em attractor} of $X$ is transitive set $A$ for which there is a compact neighborhood $U$ satisfying
$$
A=\displaystyle\bigcap_{t\geq0}X_t(U).
$$
The {\em nonwandering set} $\Omega(X)$ of $X$ is defined as the set of points $x\in M$ such that
for every neighborhood $U$ of $x$ and $T>0$ there is $t>T$ satisfying
$X_t(U)\cap U\neq\emptyset$.
Clearly $\omega(x)\subset \Omega(X)\subset M(X)$ for every $x\in M$.
By a {\em singularity} of $X$ we mean a point $\sigma\in M$ satisfying $X(\sigma)=0$.

\begin{defi}
A compact invariant set $\Lambda$ of $X$ is {\em hyperbolic}
if there are a decomposition $T_\Lambda M=E^s_\Lambda\oplus E^X_\Lambda\oplus E^u_\Lambda$
of the tangent bundle over $\Lambda$ as well as positive constants $K,\lambda$
and a Riemannian metric $\|\cdot\|$ on $M$ satisfying
\begin{enumerate}
\item
$\|DX_t(x)/E^s_x|\leq Ke^{-\lambda t}$, for every $x\in \Lambda$ and $t\geq0$.
\item
$E^X_\Lambda$ is the subbundle generated by $X$.
\item
$m(DX_t(x)/E^u_x)\geq K^{-1}e^{\lambda t}$, for every $x\in\Lambda$ and $t\geq0$ where
$m(\cdot)$ indicates the conorm operation.
\end{enumerate}
If $E^s_x\neq0$ and $E^u_x\neq0$ for all $x\in \Lambda$ we will say that $\Lambda$
is a {\em saddle-type hyperbolic set}.
A {\em hyperbolic attractor} is an attractor which is simultaneously a hyperbolic set.
A singularity  $\sigma$ of $X$ is {\em hyperbolic} if it is hyperbolic as a compact invariant set, or,
equivalently, if the linear map $DX(\sigma)$ has no purely imaginary eigenvalues.
\end{defi}

\begin{defi}[\cite{memo}]
A compact invariant set $\Lambda$ of $X$ is {\em sectional-hyperbolic}
if there are a decomposition $T_\Lambda M=E^s_\Lambda\oplus E^c_\Lambda$
of the tangent bundle over $\Lambda$ as well as positive constants $K,\lambda$
and a Riemannian metric $\|\cdot\|$ on $M$ satisfying
\begin{enumerate}
\item
$\|DX_t(x)/E^s_x\|\leq Ke^{-\lambda t}$ for every $x\in \Lambda$ and $t\geq 0$.
\item
$\frac{\|DX_t(x)/E^s_x\|}{m(DX_t(x)/E^c_x)}\leq Ke^{-\lambda t}$, for every $x\in\Lambda$ and
$t\geq0$.
\item
$
|\det(DX_t(x)/L_x)|\geq K^{-1}e^{\lambda t}
$
for every $x\in \Lambda$, $t\geq0$ and every two-dimensional subspace $L_x$ of $E^c_x$.
\end{enumerate}
\end{defi}

\begin{defi}[\cite{m1}]
A {\em sectional-Anosov flow} is a vector field whose maximal invariant set is sectional-hyperbolic.
\end{defi}

The following definition describes a certain type of singularities for sectional-Anosov flows.

\begin{defi}
We say that a singularity $\sigma$ of a vector field $X$ on a $3$-manifold $M$ is {\em Lorenz-like}
if, up to some order, the eigenvalues $\{\lambda_1,\lambda_2,\lambda_3\}$ of $DX(\sigma):T_\sigma M\to T_\sigma M$
satisfy the eigenvalue condition $\lambda_2<\lambda_3<0<-\lambda_3<\lambda_1$.
\end{defi}

A sectional-Anosov flow with singularities of a compact $3$-manifold may have Lorenz-like singularities or not \cite{bm}, \cite{bm3}, \cite{mmm}.
With these definitions we can state our main theorem.

\begin{theorem}
\label{thAA}
For every sectional-Anosov flow of a compact $3$-manifold $M$ there is a finite collection of hyperbolic attractors and Lorenz-like
singularities whose basins form a dense subset of $M$.
\end{theorem}

Applying this result we obtain easily the equivalence below.

\begin{corollary}
\label{coro2}
A sectional-Anosov flow $X$ of a compact $3$-manifold $M$ is essentially hyperbolic
if and only if the basin of the set of singularities of $X$ is nowhere dense in $M$.
\end{corollary}

Examples of sectional-Anosov flows
for which the properties of the above corollary fail are the geometric Lorenz attractors.
Further examples are the Anosov flows, the hyperbolic attracting sets (both without singularities) and
the ones in \cite{mmm}.
We also observe that there are sectional-Anosov flows on certain compact $3$-manifolds which are essentially hyperbolic
but not Axiom A: They can be obtained by modifying the singular horseshoe in \cite{lp}.

For the next corollary we shall use the following classical definition.

\begin{defi}
We say that a vector field $X$ of a manifold $M$ is {\em sensitive to the initial conditions} if there is $\delta>0$ such that
for every $x\in M$ and every neighborhood $U$ of $x$ there are $y\in U$ and
$t\geq 0$ such that $d(X_t(x),X_t(y))>\delta$.
The number $\delta$ will be referred to as a {\em sensitivity constant} of $X$
\end{defi}

This is a basic property of chaotic systems widely studied in the literature
\cite{bbcds}, \cite{g},\cite{l}, \cite{pt}, \cite{p}, \cite{r}, \cite{s}.
The following corollary asserts that this property holds for all sectional-Anosov flows on compact $3$-manifolds.
More precisely, we have the following result.

\begin{corollary}
\label{thA}
Every sectional-Anosov flow of a compact $3$-manifold is sensitive to the initial conditions.
\end{corollary}

To finish we state a topological consequence of Theorem \ref{thA}.
Recall that a compact $3$-manifold $M$ is {\em atoroidal} if
every two-sided embedded torus $T$ on $M$, for which the homeomorphism of fundamental groups
$\pi_1(T)\to \pi_1(M)$ induced by the inclusion is injective, is isotopic to a boundary component
of $M$.

By Corollary 2.6 in \cite{m1} we have that
a sectional-Anosov flow with singularities, all Lorenz-like, but without
null homotopic periodic orbits in an atoroidal compact 3-manifold has no hyperbolic attractors.
This together with Theorem \ref{thA} implies the following corollary
yielding a relationship between topology and the denseness of the basin of the singularities.

\begin{corollary}
Let $X$ be a sectional-Anosov flow of a compact atoroidal $3$-manifold $M$.
If $X$ has singularities (all Lorenz-like)
but not null homotopic periodic orbits, then the basin of the set of singularities of $X$ is dense in $M$.
\end{corollary}

An example where the hypotheses of the above corollary are fulfilled
is the geometric Lorenz attractor.

The proof of Theorem \ref{thAA} relies on the techniques in \cite{ams}, \cite{m1} but with some important differences.
For instance, the proof in \cite{ams} is based on the {\em Property (P)}
that the unstable manifold of every periodic point of $X$ intersects
the stable manifold of a singularity of $X$.
This property not only holds for every vector field close to a nonwandering sectional-Anosov flow
of a compact $3$-manifold, but also implies that the basin of the singularities of $X$ is dense in $M$.

In our case we do not have this property since the vector fields
under consideration are not close to a nonwandering sectional-Anosov flow in general.
To bypass this problem we will prove that every sectional-Anosov flow
comes equipped with a positive constant $\delta$
such that every point whose omega-limit set passes $\delta$-close to some singularity
is accumulated by the stable manifolds of the singularities.
To prove this assertion we combine some arguments from \cite{ams}, \cite{bm1} and \cite{m2}.
This assertion is the key ingredient for the proof of Theorem \ref{thAA}. Corollary \ref{thA} will be obtained easily from this theorem and Lemma \ref{dense}.
Both results will be proved in the last section.

\section{Preliminars}

\noindent
In this section we prove some lemmas which will be used to prove our results.
We start with some basic definitions.
Let $X$ be a $C^1$ vector field on $M$ inwardly transverse to $\partial M$ (if $\partial M\neq\emptyset$).
For every $x\in M(X)$ we define the sets
$$
W^{ss}(x)=\{y\in M:d(X_t(x),X_t(y))\to 0\mbox{ as }t\to\infty\},
$$
$$
W^{uu}(x)=\{y\in M:d(X_t(x),X_t(y))\to 0\mbox{ as }t\to-\infty\}
$$
$$
W^{s}(x)=\displaystyle\bigcup_{t\in \mathbb{R}}W^{ss}(X_t(x))
\quad\mbox{ and }\quad
W^{u}(x)=\displaystyle\bigcup_{t\in \mathbb{R}}W^{uu}(X_t(x))
$$
We denote by $Sing(X)$ the set of singularities of $X$
and denote by
$$
W^s(Sing(X))=\displaystyle\bigcup_{\sigma\in Sing(X)}W^s(\sigma)
$$
the basin of $Sing(X)$.
We say that a point $p$ is {\em periodic} for $X$ if
there is a minimal $t>0$ such that $X_t(p)=p$. Denote by $Per(X)$ the set of periodic points of $X$.
We shall use the following  auxiliary definition.

\begin{defi}
\label{good}
An {\em intersection number} for a vector field $X$ is a positive number $\delta$
such that if $p\in Per(X)$ and $W^u(p)\cap B_\delta(Sing(X))\neq\emptyset$,
then $W^u(p)\cap W^s(Sing(X))\neq\emptyset$.
\end{defi}

Applying the connecting lemma \cite{bm2} as in Lemma 1 of \cite{m2}
we obtain the following key fact.

\begin{lemma}
\label{thedelta}
Every sectional-Anosov flow of a compact $3$-manifold has an intersection number.
\end{lemma}

Let $X$ be a vector field on a compact manifold $M$ inwardly transverse to $\partial M$ (if $\partial M\neq\emptyset$).
Given $\delta>0$ we define
\begin{equation}
\label{the set}
H_\delta=\displaystyle\bigcap_{t\in\mathbb{R}}X_t(M\setminus B_\delta(Sing(X))).
\end{equation}
If $X$ is sectional-Anosov,
then $H_\delta$ is a saddle-type hyperbolic set \cite{bm}, \cite{mpp}.

As a first application of Lemma \ref{thedelta} we obtain the following corollary (which is also true
in higher dimensions).

\begin{corollary}
\label{**}
The number of attractors of a sectional-Anosov flow on a compact $3$-manifold is finite.
\end{corollary}

\begin{proof}
Suppose by contradiction that there is a sectional-Anosov flow
of a compact $3$-manifold exhibiting an infinite sequence of attractors $A_k$, $k\in \mathbb{N}$.
Since the family of attractors of $X$ is pairwise disjoint, and
$Sing(X)$ is finite, we can assume that none of such attractors have a singularity.
By Lemma \ref{thedelta} we can fix an intersection number $\delta$ of $X$.
If one of the attractors $A_k$ intersects $B_\delta(Sing(X))$, then
we can select a periodic point $p_k\in A_k\cap B_\delta(Sing(X))$.
Since $\delta$ is an intersection number we would have that
$W^u(p_k)\cap W^s(Sing(X))\neq\emptyset$. But $W^u(p_k)\subset A_k$ (for $A_k$
is an attractor) so $A_k$ contains a singularity, a contradiction.
Therefore, $B_\delta(Sing(X))\cap \left(\bigcap_{k}A_k\right)=\emptyset$ and so
$\bigcup_{k}A_k\subset H_\delta$ where $H_\delta$ is given in (\ref{the set}).
Since $X$ is sectional-Anosov we have that $H_\delta$ is a hyperbolic set and, since the numbers of attractors on a hyperbolic set is finite,
we obtain that the sequence $A_k$ is finite, a contradiction.
This ends the proof.
\end{proof}

Next we recall the terminology of singular partitions \cite{bm1}.

Consider a vector field $X$ on a compact manifold $M$.
By a {\em cross section} of $X$ we mean a codimension one submanifold $\Sigma$
which is transverse to $X$. The interior and the boundary of $\Sigma$ (as a submanifold)
will be denoted by $Int(\Sigma)$ and $\partial\Sigma$ respectively.
Given a family of cross sections $\mathcal{R}$ we still denote by
$\mathcal{R}$ the union of its elements. We also denote
$$
\partial \mathcal{R}=\displaystyle\bigcup_{\Sigma\in \mathcal{R}}\partial\Sigma
\quad\mbox{ and }
\quad
Int(\mathcal{R})=\displaystyle\bigcup_{\Sigma\in \mathcal{R}}Int(\Sigma).
$$

\begin{defi}
A {\em singular partition} of a compact invariant set $\Lambda$ of $X$ is a finite disjoint collection
of cross sections $\mathcal{R}$ satisfying
$$
\Lambda\cap\partial\mathcal{R}=\emptyset\quad\mbox{ and }\quad
Sing(X)\cap \Lambda=\{x\in \Lambda:X_t(x)\not\in\mathcal{R},\forall t\in \mathbb{R}\}.
$$
\end{defi}

A cross section $\Sigma$ of $X$ is a {\em rectangle} if it is diffeomorphic to
$[0,1]\times [0,1]$. In this case the boundary $\partial \Sigma$ is formed by
two vertical curves, with union $\partial^v\Sigma$, and two horizontal curves.
If $z\in Int(\Sigma)$ we say that the rectangle $\Sigma$ is around $z$.
On the other hand, it is well known from the invariant manifold theory \cite{hps} that the
subbundle $E^s$ of a sectional-Anosov flow $X$ can be integrated yielding a strong
stable foliation $W^{ss}$ on $M$. As usual we denote by $W^{ss}(x)$ the leaf of this foliation passing
through $x\in M$. In the case when $x\in \Sigma$ we denote by
$\mathcal{F}^s(x,\Sigma)$ (or simply $\mathcal{F}^s(x)$) the projection of $W^{ss}(x)$
onto $\Sigma$ along the orbits of $X$.

For any set $A$ we denote by $Cl(A)$ its closure and by $B_\delta(A)$ (for $\delta>0$) we denote
the $\delta$-ball centered at $A$.

The following lemma uses intersection numbers to find singular partitions for certain omega-limit sets. Its proof follows closely that of Theorem 3 in \cite{ams}.

\begin{lemma}
\label{size}
Let $\delta$ be an intersection number of a sectional-Anosov flow on a compact $3$-manifold $M$. If
$x\notin Cl(W^s(Sing(X)))$ satisfies $\omega(x)\cap B_\delta(Sing(X))\neq\emptyset$,
then $\omega(x)$ has a singular partition.
\end{lemma}

\begin{proof}
By Proposition 2 in \cite{ams} it suffices to prove that for every $z\in \omega(x)\setminus Sing(X)$
there is a cross section $\Sigma_z$ through $z$ such that
$\omega(x)\cap \partial \Sigma_z=\emptyset$.
So fix $z\in \omega(x)\setminus Sing(X)$.

We claim that $\omega(x)\cap W^{ss}(z)$ has empty interior in $W^{ss}(z)$.
If $\omega(x)$ has a singularity this follows from the Main Theorem of \cite{m3}
(indeed, the proof in \cite{m3} was done for transitive sets but works
for omega-limit sets also).
Then, we can assume that $\omega(x)$ has no singularities, and so, it is a hyperbolic
set \cite{bm}, \cite{mpp}.
If $\omega(x)\cap W^{ss}(z)$ has nonempty interior in $W^{ss}(z)$ then
$\omega(x)$ contains a local strong stable manifold $W^{ss}_\epsilon(y)$ for some
$y\in \omega(x)$. From this and the hyperbolicity of $\omega(x)$
we obtain $x\in \omega(x)$ and so $x\in \Omega(X)$.
As $x\notin Cl(W^s(Sing(X)))$ the closing lemma in \cite{m2} implies that
there is a sequence $p_n\in Per(X)$ converging to $x$.
As $\omega(x)\cap B_\delta(Sing(X))\neq\emptyset$ the above convergence implies that
the orbit of $p_n$ intersects $B_\delta(Sing(X))$ for $n$ large.
As $\delta$ is an intersection number we obtain
$W^u(p_n)\cap W^s(Sing(X))\neq\emptyset$ for all $n$ large.
From this and the Inclination Lemma \cite{kh} we obtain that $x\in Cl(W^s(Sing(X)))$ which is absurd.
The claim follows.

Using the claim we obtain a rectangle $R_z$ around $z$ such that
$\omega(x)\cap \partial^v R_z=\emptyset$.
If the positive orbit of $x$ intersects only one component of $R_z\setminus \mathcal{F}^s(z)$
we select a point $x'$ in that component and a point $z'$ in the other component.
In such a case we define $\Sigma_z$ as the subrectangle of
$R_z$ bounded by $\mathcal{F}^ s(x')$ and $\mathcal{F}^s(z')$.
We certainly have that $\omega(x)\cap \mathcal{F}^s(z')=\emptyset$.
On the other hand, if $\omega(x)$ intersects $\mathcal{F}^s(x')$ in a point $h$ (say)
then $h\in \Omega(X)$ and so, by the closing lemma \cite{m2}, $h$ is approximated by periodic
points or by points whose omega-limit set is a singularity.
The latter option must be excluded (for it would imply $x\in Cl(W^s(Sing(X)))$)
so $h$ is the limit of a sequence $p_n\in Per(X)$.
But $h\in \mathcal{F}^ s(x')$ so $\omega(h)=\omega(x')$.
As $x'$ and $x$ belongs to the same orbit we also have $\omega(x')=\omega(x)$
yielding $\omega(h)=\omega(x)$. As $\omega(x)\cap B_\delta(Sing(X))\neq\emptyset$
we obtain $\omega(h)\cap B_\delta(Sing(X))\neq\emptyset$ too.
Since $p_n$ approaches $h$ we conclude that the orbit of $p_n$ intersects $B_\delta(Sing(X))$
for all $n$ large.
As $\delta$ is an intersection number we obtain that $W^u(p_n)\cap W^s(Sing(X))\neq\emptyset$
and so $h\in Cl(W^s(Sing(X)))$ by the Inclination Lemma as before.
From this and the uniform size of the stable manifolds we obtain $x\in Cl(W^s(Sing(x)))$
which is a contradiction.
Therefore, $\omega(x)\cap \mathcal{F}^s(x')=\emptyset$.
As $\partial \Sigma_z$ consists of $\partial^v\Sigma_z$ together
with $\mathcal{F}^s(x' )\cup\mathcal{F}^s(z')$ we obtain
$\omega(x)\cap \partial \Sigma_z=\emptyset$.
The construction of $\Sigma_z$ is similar in the case when $\omega(x)$ intersects both components
of $R_z\setminus \mathcal{F}^s(z)$. This finishes the proof.
\end{proof}

A second auxiliary definition is as follows.

\begin{defi}
\label{cdelta}
Let $X$ be a vector field of a compact manifold $M$. Given $\delta\geq0$ we say that $X$ satisfies (C)$_\delta$ if
$$
\{x\in M:\omega(x)\cap B_\delta(Sing(X))\neq\emptyset\}\subset Cl(W^s(Sing(X))).
$$
\end{defi}

We shall use the previous lemma to prove the following result. Its proof
follows closely that of Theorem 4 in \cite{ams}.

\begin{lemma}
\label{ass2}
If $\delta$ is an intersection number of a sectional-Anosov flow $X$ of a compact $3$-manifold $M$,
then $X$ satisfies (C$)_\delta$.
\end{lemma}

\begin{proof}
Suppose by contradiction that (C$)_\delta$ fails.
Then, there is $x\in M$ such that $\omega(x)\cap B_\delta(Sing(X))\neq\emptyset$
but $x\notin Cl(W^s(Sing(X)))$.

By  Lemma \ref{size} we have that
$\omega(x)$ has a singular partition $\mathcal{R}$.
Using that $x\notin Cl(W^s(Sing(X)))$ we can choose an interval $I$ around $x$,
tangent to $E^c_x$, which does not intersect $W^s(Sing(X))$.
On the other hand,
we have clearly that $\omega(x)$ is not a singularity.
Then, by Theorem 11 in \cite{bm1},
there are $S\in R$,  sequence $x_n\in S$ (in the positive orbit of $x$),
a sequence $I_n$ of intervals around $x_n$ (in the positive orbit of $I$)
such that both components of $I_n\setminus \{x_n\}$ have length bounded away from zero.
We can assume that $x_n\to w$ for some $w\in S$ and further
$I_n$ converges to an interval $J$ around $w$ tangent to $E^c_w$.
But $w\in \omega(x)$ so $w\in \Omega(X)$ and, then, by the closing lemma \cite{m2},
we have that $w$ is accumulated by periodic points or by points whose omega-limit set is a singularity.
In the latter case we have from the uniform size of the stable manifolds that
$J\cap W^s(Sing(X))\neq\emptyset$ and so
$J_n\cap W^s(Sing(X))\neq\emptyset$ for all $n$ large.
As $J_n$ belongs to the orbit of $I$ we obtain
$I\cap W^s(Sing(X))\neq	\emptyset$ which is a contradiction.
Therefore, there is a sequence $p_n\in Per(X)$ converging to $w$.
If $W^u(p_n)\cap B_\delta(Sing(X))\neq\emptyset$ for infinitely many $n$'s we obtain from the fact that $\delta$ is an intersection number
that $W^u(p_n)\cap W^s(Sing(X))\neq\emptyset$ for such integers $n$.
Applying the Inclination Lemma as before we obtain that $x\in Cl(W^s(Sing(X)))$, a contradiction.
From this we conclude that
$Cl(W^u(p_n))\cap B_\delta(Sing(X))=\emptyset$ for all $n$ large.
In particular, every $p_n$ belongs to
the set $H_\delta$ defined in (\ref{the set}) which is hyperbolic, and so,
the unstable manifold $W^u(p_n)$ has uniformly large size for large $n$.
As both $p_n$ and $x_n$ converges to $w$ we obtain that
there is a point in the positive orbit of $x$ whose omega-limit set is contained
in $Cl(W^u(p_n))$ for some $n$. It follows that $\omega(x)\subset Cl(W^u(p_n))$
and, then, $Cl(W^u(p_n))\cap B_\delta(Sing(X))\neq\emptyset$  which is absurd.
This contradiction concludes the proof.
\end{proof}

To prove Corollary \ref{thA} we need the following generalization of Proposition 1 in \cite{ams}.

\begin{lemma}
\label{dense}
Every vector field $X$ of a compact manifold $M$
exhibiting a finite collection of saddle-type hyperbolic attractors and singularities
whose basins form a dense subset of $M$ is
sensitive to the initial conditions.
\end{lemma}

\begin{proof}
If $X$ has no saddle-type hyperbolic attractors, then the result follows from Proposition 1 in \cite{ams}.
So, we can assume that $X$ has at least one saddle-type hyperbolic attractor. 

Let $\{A_1,\cdots, A_r\}$ and $\{\sigma_1,\cdots, \sigma_l\}$ be the collection of saddle-type hyperbolic attractors and singularities of $X$
whose basins form a dense subset of $M$.
As is well-known (p.9 in \cite{pt}) for every $i=1,\cdots, r$ there is $\beta_i>0$
such that $X$ restricted to $B_{\beta_i}(A_i)$ is sensitive to the initial conditions. Let $\delta_i$ be the corresponding sensitivity constant for $i=1,\cdots, r$.

To conclude the proof we shall prove that any positive number $\delta$ less than
$$
\min\left\{\frac{\beta_1}{2},\cdots, \frac{\beta_r}{2},\delta_1,\cdots,\right.$$
$$
\left.
\delta_r,\min\{d(B,C):B,C\in
\{A_1,\cdots, A_r,\sigma_1,\cdots, \sigma_l\}, B\neq C\}\right\}
$$
is a sensitivity constant of $X$.

Indeed, take $x\in M$ and suppose by contradiction that there is a neighborhood $U$ of
$x$ such that $d(X_t(x),X_t(y))\leq \delta$ for every $t\geq0$.
Suppose for a while that there is $y\in U$ such that $\omega(y)\subset A_i$ for some $i=1,\cdots, r$.
Then, $d(X_t(y), A_i)<\frac{\beta_i}{2}$ for some $t\geq 0$ so
$$
d(X_t(x),A_i)\leq d(X_t(x),X_t(y))+d(X_t(y),A_i)\leq \delta+\frac{\beta_i}{2}<\frac{\beta_i}{2}+\frac{\beta_i}{2}=\beta_i
$$
thus $X_t(x)\in B_{\beta_i}(A_i)$. From this we can find $T\geq t$ and also $z\in U$ such that
$d(X_T(x),X_T(z))\geq \delta_i\geq \delta$.
Therefore, we can assume that there is no $y\in U$ within the union of the basins of the attractors
$\{A_1,\cdots, A_r\}$, and so,
$W^s(\{\sigma_1,\cdots, \sigma_l\})\cap U$ is dense in $U$ by the hypothesis.
Now we can proceed as in the proof of Proposition 1 in \cite{ams} to obtain the desired contradiction.

More precisely, we have two possibilities, namely, either $x\in W^s(\sigma_i)$ for some $i=1,\cdots, l$ or not.
In the first case we can select $y\in U$ outside $W^s(\sigma_i)$ since $W^s(\sigma_i)$ has no interior (recall $\sigma_i$ is saddle-type).
Since the positive orbit of $x$ converges to $\sigma_i$, and that of $y$ does not, we eventually find $t > 0$
such that $d(X_t(x),X_t(y))\geq \delta$ which is absurd.
In the second case can use the hypothesis to select $y\in W^s(\sigma_i)\cap U$ for some $i=1,\cdots, l$ since
$W^s(\{\sigma_1,\cdots, \sigma_l\})\cap U$ is dense in $U$. Again we argue that since the
positive orbit of $y$ converges to $\sigma_i$, and that of $x$ does not, we eventually find $t > 0$ such that
$d(X_t(x),X_t(y))\geq \delta$ which is absurd too.
These contradictions prove that $\delta$ as above is a sensitivity constant of $X$ and the result follows.
\end{proof}

\section{Proof of Theorem \ref{thAA} and Corollary \ref{thA}}

\begin{proof}[Proof of Theorem \ref{thAA}]
Let $X$ be a sectional-Anosov flow of a compact $3$-manifold.
By Lemma \ref{thedelta} we have that $X$ has an intersection number $\delta$ and by Lemma \ref{ass2} we have that
$X$ satisfies (C$)_\delta$. For such a $\delta$ we let $H_\delta$ be as in (\ref{the set}).

Now take $x\not\in Cl(W^s(Sing(X)))$. Since $Cl(W^s(Sing(X)))$ is closed
there is a neighborhood $U$ of $x$ such that
$U\cap Cl(W^s(Sing(X)))=\emptyset$. By (C$)_\delta$ we have
$\omega(y)\cap B_\delta(Sing(X))=\emptyset$ and then
$\omega(y)\subset H_\delta$ for every $y\in U$. But
$H_\delta$ is hyperbolic, so, there is an open and dense subset of
$U$ all of whose points belong to the basin of a hyperbolic attractor of $X$ in $H_\delta$.
This proves that the union of the basins of the hyperbolic attractors form together with
$W^s(Sing(X))$ a dense subset of $M$. As the union of the stable manifolds of the non-Lorenz-like singularities is nowhere dense (c.f. \cite{bm}, \cite{bm3})
we obtain that the union of the basins of the hyperbolic attractors and the Lorenz-like singularities is dense in $M$.
As $X$ has only a finite number of both hyperbolic attractors (by Corollary \ref{**}) and Lorenz-like singularities (for they are hyperbolic) we are done.
\end{proof}

\begin{proof}[Proof of Corollary \ref{thA}]
By Theorem \ref{thAA} we have that every sectional-Anosov flow of a compact $3$-manifold
satisfies the hypotheses of Lemma \ref{dense}.
So, the result follows from this lemma.
\end{proof}

\end{document}